\documentclass[12pt]{article}

\usepackage{amssymb}
\usepackage{amsfonts}
\usepackage{theorem}
\usepackage{bbm}
\usepackage{graphicx}
\usepackage[leqno]{amsmath}
\usepackage{enumerate}
\usepackage{indentfirst}
\usepackage{color}

\newtheorem{thm}{Theorem}

\newtheorem{remark}{Remark}
\newtheorem{lmm}{Lemma}[section]
\newtheorem{lemma}{Lemma}[section]

\newtheorem{corollary}{Corollary}%[section]

\newenvironment{proof}[1][Proof]{\textbf{#1.} }{\hfill $\square$}

\newcommand{\E}{\mathbb{E}}
\newcommand{\R}{\mathbbm{R}}

\newcommand{\eps}{\varepsilon}

\newcommand{\mmI}{\mathbf{I}}

\newcommand{\cL}{\mathcal{L}}

\newcommand{\cA}{\mathcal{A}}

1

\title{Homogenization of random parabolic operators. Diffusion approximation. }
\author{
M. Kleptsyna \thanks{Universit\'e du
Maine, D\'epartement de Math\'ematiques, Laboratoire Manceau de Math\'ematiques, Avenue Olivier Messiaen, 72085 Le Mans, Cedex 9, France.\hfill\break
e-mail: {\tt marina.kleptsyna@univ-lemans.fr}\hfill\break
e-mail: {\tt alexandre.popier@univ-lemans.fr}
}
\,\,\, A. Piatnitski \thanks{Faculty of Technology, Narvik University College, Norway \  and \ Lebedev Physical Institute RAS, Moscow, Russia\hfill\break
e-mail: \ {\tt andrey@sci.lebedev.ru}
\hfill\break
\vskip -3mm
The work of the first and the third authors was partially supported by ANR STOSYMAP
}
\,\,\, and A. Popier\footnotemark[2]
}

\date{\today}

\begin{document}
\maketitle

\begin{abstract}
The paper deals with homogenization of divergence form
second order parabolic operators whose coefficients are
periodic in spatial variables and random stationary in time.
Under proper mixing assumptions, we study the limit behaviour
of the  normalized difference between solutions of the original and the homogenized problems. The asymptotic behaviour of this
difference depends crucially on the ratio between spatial
and temporal scaling factors. Here we study the case of self-similar parabolic diffusion scaling.

\end{abstract}

%\footnote{}

\section{Introduction}
%------------------

The goal of this paper is to characterize the rate of convergence in the homogenization problem for a second
order divergence form parabolic operator with random stationary in time and periodic in spatial variables
coefficients. We are also aimed at describing the limit behaviour of a normalized difference between
solutions of the original and homogenized problems
%In the paper we consider the homogenization problem for a second order divergence form parabolic operators.

To avoid boundary effects we study a Cauchy problem that takes the form
\begin{equation}\label{ory_cauchy}
\begin{array}{c}
\displaystyle
\partial_t u^\eps= \mathrm{div}\Big(a\big(\frac{x}{\eps},\frac{t}{\eps^\alpha}\big)\nabla u^\eps\Big),
\qquad x\in\mathbb R^n, \ t>0,
\\[3mm]
\displaystyle
u^\eps(x,0)=g(x).
\end{array}
\end{equation}
with $\alpha>0$. In this paper we consider the case $\alpha=2$.
We assume that the matrix $a(z,s)=\{a^{ij}(z,s)\}$ is uniformly elliptic, $(0,1)^n$-periodic in $z$ variable, and random stationary ergodic in $s$. We denote $Y=(0,1)^n$ and in what follows identify $Y$-periodic function with functions define on the torus $\mathbb T^n$.
%We also assume that $\alpha>0$.

It is known (see \cite{JKO82}, \cite{KP02}) that under these assumptions
problem (\ref{ory_cauchy}) admits homogenization. More precisely, for any $g\in L^2(\mathbb R^n)$, almost surely (a.s.) solutions $u^\eps$ of problem (\ref{ory_cauchy})
converge, as $\eps\to 0$, to a solution
of the homogenized problem
\begin{equation}\label{hom_cauchy}
\begin{array}{c}
\displaystyle
\partial_t u^0=\mathrm{div}\big(a^{\rm eff}\nabla u^0\big)\\[2mm]
\displaystyle
u^0(x,0)=g(x)
\end{array}
\end{equation}
with a constant non-random coefficients.
The convergence is in $L^2(\mathbb R^n\times(0,T))$.
More detailed description of the existing homogenization results is given in Sections \ref{s_homoge} and \ref{s_diffus}.

The paper focuses on the rate of this convergence and on higher order terms of the asymptotics of $u^\eps$. Our goal is to describe the limit behaviour of the normalized difference $\eps^{-1}(u^\eps-u^0)$.

Clearly, the main oscillating term of the asymptotics of this
normalized difference should be expressed in terms of the corrector. We recall (see \cite{KP02}, \cite{DIPP06})  that
the equation
$$
\partial_s\chi(z,s)=\mathrm{div}_z\big(a(z,s)\big(\nabla_z \chi(z,s)+{\bf I}\big)\big)
$$
has a unique up to an additive (random) constant periodic in $z$ and stationary in $s$ solution. Thus, the gradient $\nabla_z\chi$ is uniquely defined. The principal corrector takes the form
$\eps\chi\big(\frac{x}{\eps},\frac{t}{\eps^2}\big)\cdot\nabla u^0(x,t)$. We study the limit behaviour of the expression
$$
U^\eps(x,t):= \frac{u^\eps(x,t)-u^0(x,t)}{\eps}- \chi \big(\frac{x}{\eps},\frac{t}{\eps^2}\big)\cdot\nabla u^0(x,t).
$$
For generic stationary ergodic coefficients $a(z,s)$ the family $\{U^\eps\}$ need not be compact or tight in $L^2(\mathbb R^n\times(0,T))$.

For this reason we assume that (see Section \ref{s_setup} for further details)
\begin{itemize}
\item Coefficients $a(z,s)$ have good mixing properties.
\item Initial function $g$ is sufficiently smooth.
\end{itemize}
Under these conditions we show (see Theorem \ref{t_main1}, Section \ref{sec6})
%\ref{lim_spde}
that $U^\eps$ converges in law
in $L^2(\mathbb R^n\times(0,T))$ equipped with the strong topology to a solution of a SPDE with constant coefficients and an additive noise. This SPDE reads
$$
\begin{array}{c}
\displaystyle
d U^0=\mathrm{div}\Big(a^{\rm eff}\nabla U^0+\mu \frac{\partial^3}{\partial x^3}u^0\Big)dt+\Lambda^{1/2}\frac{\partial^2}{\partial x^2}u^0\,dW_t,\\[4mm]
U^0(x,0)=0;
\end{array}
$$
here $a^{\rm eff}$ is the homogenized coefficients matrix, $u^0$ is a solution of (\ref{hom_cauchy}),
$W_t$ is a standard $n^2$-dimensional Wiener process, and $\mu$ and $\Lambda$ are constant tensors
which are defined in Section \ref{sec6}.
We show that this SPDE is well-posed and, thus, defines
the limit law of $U^\eps$ uniquely.

Notice that under proper choice of an additive constant the mean value of  $\chi(z,s)$ is equal to zero.
Therefore, the function $ \chi \big(\frac{x}{\eps},\frac{t}{\eps^2}\big)\nabla u^0(x,t)$ converges a.s. to zero
weakly in $L^2(\mathbb R^n\times(0,T))$, as $\eps\to 0$. Therefore, in the weak topology of $L^2(\mathbb R^n\times(0,T))$, the limit in law of the normalized difference
 $\eps^{-1}(u^\eps(x,t)-u^0(x,t))$ coincides with that of $U^\eps$.

The first results on homogenization of elliptic operators
with random statistically homogeneous coefficients were obtained in \cite{Ko78}, \cite{PaVa81}. At present there is an extensive literature on this topic. However, optimal estimates for the rate of convergence is an open issue. In \cite{Yur81} some power estimates for the rate of convergence were obtained in dimension three and more.
In the recent work \cite{GlOt12} the further important progress
has been made in this problem.

Parabolic operators with random coefficients depending both on spatial and temporal variables have been considered in \cite{JKO82}. In the case of a diffusive scaling, the a.s. homogenization theorem has been proved.

The case of non-diffusive scaling has been studied in \cite{KP97} under the assumption that the coefficients are periodic in spatial variables and random stationary in time.

It turns out that the structure of the higher order terms
of the asymptotics of $u^\eps$ depends crucially on whether
the scaling is diffusive or not. Here we study the diffusive scaling. The case of non-diffusive scaling will be addressed elsewhere.

%establish the limiting development for the singulary
%perturbed parabolic operator with random coefficients,
%rapidly oscillating in both space and time variables.
%We consider the periodic microscopic structure while the
%dynamic of the system is random, generated by the stationary
%ergodic process.

\section{The setup} \label{s_setup}
%------------------

Let $(\Omega,\mathcal{F},{\bf P})$ be a standard probability space equipped with a measure preserving ergodic dynamical system $\mathcal{T}_s$, $s\in\mathbb R$.

Given a measurable matrix function $\tilde a(z,\omega)=\{\tilde a^{ij}(z,\omega)\}_{i,j=1}^n$ which is periodic in $z$ variable with a period one in each coordinate direction, we define
a random field $a(z,s)$ by
$$
a(z,s)=\tilde a(z,\mathcal{T}_s\omega).
$$
Then $a(z,s)$ is periodic in $z$ and stationary ergodic  in $s$.

We consider the following Cauchy problem
in $\mathbb R^n\times(0,T]$, $T>0$\,:
\begin{equation} \label{eq:u^eps}
\left\{ \begin{array}{lcl} \displaystyle
\frac{\partial u^\eps}{\partial t}& = &\displaystyle \mathrm{div}\Big( a\Big( \frac{x}{\eps}, \frac{t}{\eps^2} \Big)\nabla u^\eps \Big)   \\[4mm]
u^\eps (0,x) & = & g(x)
\end{array} \right.
\end{equation}
with a small positive parameter $\eps$.

%, $\alpha>0$ being a fixed  number.
We assume that the coefficients in (\ref{eq:u^eps}) possess the following properties.
\begin{itemize}
\item[{\bf H1}]
The matrix $a(z,s)$ is symmetric and satisfies uniform ellipticity conditions that is there is $\lambda>0$ such that for all $(z,\omega)$ the following inequality holds\,:
$$
\lambda|\zeta|^2\le \tilde a(z,\omega)\xi\cdot\xi\le \lambda^{-1}|\zeta|^2\qquad \hbox{for all }\zeta\in \mathbb R^n.
$$
\item[{\bf H2}] The initial condition $g$ is four times continuously differentiable, and for any $K>0$ there is $C_K>0$ such that
$$
\sum\limits_{|{\bf  j}|=0}^4 \Big|\frac{\partial^{\bf  j}}{(\partial x^1)^{j_1}\dots (\partial x^n)^{j_n}} g(x)\Big|\le C_K(1+|x|)^{-K}
$$
for all $x\in\mathbb R^n$, and ${\bf j}=(j_1,\ldots,j_n)$
\end{itemize}
In order to formulate one more condition we introduce the so-called maximum correlation coefficient. Setting
$\mathcal{F}_{\leq r}=\sigma\{a(z,s)\,:\,s\leq r\}$ and
$\mathcal{F}_{\geq r}=\sigma\{a(z,s)\,:\,s\geq r\}$, we define
$$
\rho(r)=\sup\limits_{\xi_1,\xi_2}{\bf E}(\xi_1\xi_2)
$$
where the supremum is taken over all $\mathcal{F}_{\leq 0}$-measurable $\xi_1$ \, and $\mathcal{F}_{\geq r}$-measurable
$\xi_2$ such that ${\bf E}\xi_1={\bf E}\xi_2=0$, and  ${\bf E}\{(\xi_1)^2\}={\bf E}\{(\xi_2)^2\}=1$.  We then assume that
\begin{itemize}
\item[{\bf H3}] The function $\rho$ satisfies the estimate
$
\displaystyle \int_0^\infty\rho(r)dr<+\infty.
$
\end{itemize}

\begin{remark}
Condition {\bf H3} is somehow implicit.
In applications various sufficient conditions are often used.
In particular,  {\bf H3} is fulfilled if $\rho(r)\le cr^{-(1+\delta)}$ for some $\delta>0$.\\
\end{remark}

\begin{remark}
In an important particular case we set
$$
a(z,s)=\widetilde{\widetilde a}(z,\xi_s),
$$
where $\xi_s$ is a stationary  process with values in $\mathbb R^N$, and $\widetilde{\widetilde a}(z,y)$ satisfies the uniform ellipticity conditions
$$
\lambda|\zeta|^2\le\widetilde{\widetilde a}(z,y)\xi\cdot\xi\le \lambda^{-1}|\zeta|^2\qquad \hbox{for all }\zeta\in \mathbb R^n,\ \ (z,y)\in\mathbb Z^n
\times\mathbb R^N.
$$
If $\xi_s$ is Gaussian then condition {\bf H3} follows from integrability of the correlation function of $\xi_\cdot$.\\
If  $\xi_s$ is a diffusion process, then  condition  {\bf H3} can be replaced with some conditions on the generator
of $\xi_s$. This case is considered in Sections \ref{s_diffus} and \ref{sec7_dif}.
\end{remark}

\section{Homogenization results}\label{s_homoge}

In this section we remind of the existing homogenization results for problem (\ref{ory_cauchy}).  Although we only deal in this paper with the case $\alpha=2$, for convenience of the reader we formulate the homogenization results for all $\alpha>0$.  To this end we first introduce the so-called cell problem.
For $\alpha=2$ it reads
\begin{equation} \label{aux_alp2}
\partial_s\chi(z,s)=\mathrm{div}\big(a(z,s)({\bf I}+\nabla\chi(z,s)\big),\qquad (z,s)\in\mathbb T^n \times(-\infty,+\infty)
\end{equation}
with $\mmI$ being the unit matrix; here $\chi=\{\chi^j\}_{j=1}^n$ is a vector function. In what follows for the sake of brevity we denote
$\mathrm{div}a=\mathrm{div}(a{\bf I})=\frac{\partial}{\partial z^i}a^{ij}(z)$.  Also, we assume summation over repeated indices.

According to Lemma \ref{l_exist_stat}, under assumption {\bf H1} this equation has a stationary periodic in $y$ vector-valued solution. This solution is unique up to an additive constant.
We define
\begin{equation}\label{aeff-ravno}
a^{\rm eff}=\E \int\limits_{\mathbb T^n}a(z,s)\big(\mmI+\nabla\chi(z,s)\big)\,dz
\end{equation}
Notice that due to stationarity the expression
on the right-hand side does not depend on $s$.

If $\alpha<2$, the cell problem reads
\begin{equation} \label{aux_alp<}
\mathrm{div}\big(a(z,s)({\bf I}+\nabla\chi_-(z,s)\big)=0,\qquad z\in \mathbb T^n;
\end{equation}
here $s$ is a parameter. This equation has a unique up to a multiplicative constant solution. We then set
\begin{equation} \label{aeff-menshe}
a_{-}^{\rm eff}=\E \int\limits_{\mathbb T^n}a(z,s)\big(\mmI+\nabla\chi_-(z,s)\big)\,dz.
\end{equation}

For $\alpha>2$ we first define $\overline a(z)=\E a(z,s)$,
then introduce a deterministic function $\chi_+(z)$
as a periodic solution to the problem
\begin{equation} \label{aux_alp>}
\mathrm{div}\big(\overline a(z)({\bf I}+\nabla\chi_+(z)\big)=0,\qquad z\in \mathbb T^n,
\end{equation}
and finally define
\begin{equation} \label{aeff-bolshe}
a_{+}^{\rm eff}=\int\limits_{\mathbb T^n}\overline a(z)\big(\mmI+\nabla\chi_+(z)\big)\,dz.
\end{equation}

The following statement has been obtained in \cite{JKO82}
and \cite{DIPP06}.

\begin{thm}
Let $g\in L^2(\mathbb R^n)$, and assume that condition {\bf H1} holds.
If $\alpha=2$, then a solution $u^\eps$ of problem (\ref{ory_cauchy}) converges a.s. in $L^2(\mathbb R^n\times(0,T))$ to a solution of the limit problem (\ref{hom_cauchy}) with $a^{\rm eff}$ given by (\ref{aeff-ravno}).

If $\alpha<2$, then a solution $u^\eps$ of problem (\ref{ory_cauchy}) converges in probability in $L^2(\mathbb R^n\times(0,T))$ to a solution of the limit problem (\ref{hom_cauchy}) with $a^{\rm eff}=a_-^{\rm eff}$ defined in (\ref{aeff-menshe}).

If $\alpha>2$, then a solution $u^\eps$ of problem (\ref{ory_cauchy}) converges in probability in $L^2(\mathbb R^n\times(0,T))$ to a solution of the limit problem (\ref{hom_cauchy}) with $a^{\rm eff}=a_+^{\rm eff}$ defined in (\ref{aeff-bolshe}).
\end{thm}

\begin{remark}{\rm
An alternative way of defining the effective matrix $a^{\rm eff}$ is related to the operator with reversed time.
We define $\chi_-$ as a stationary solution of the problem
\begin{equation} \label{inv_aux}
-\partial_s\chi_-(z,s)=\mathrm{div}\big(a(z,s)({\bf I}+\nabla\chi_-(z,s)\big),\qquad (z,s)\in (-\infty,+\infty)\times\mathbb T^n
\end{equation}
and set
\begin{equation}\label{aeff-ravno-bis}
a^{\rm eff}=\E \int\limits_{\mathbb T^n}a(z,s)\big(\mmI+\nabla\chi_-(z,s)\big)\,dz
\end{equation}
In order to show that  (\ref{aeff-ravno-bis}) and  (\ref{aeff-ravno}) define the same effective matrix, we multiply the $i$-th component of equation (\ref{inv_aux}) by $\chi^j$,  and the  $j$-th component of equation (\ref{aux_alp2}) by $\chi^i$ and  integrate the resulting relations over $\mathbb T^n\times (0,1)$. Subtracting the second relation from the first one, taking the expectation and considering the symmetry of  effective matrix, we obtain the desired
equality.
}
\end{remark}

\subsection{Diffusive dependence of time}\label{s_diffus}

In this section as a particular case of (\ref{eq:u^eps}) we introduce the following problem
\begin{equation} \label{eq:u^eps_bis}
\left\{ \begin{array}{lcl}
\displaystyle
\frac{\partial u^\eps}{\partial t}& = &\displaystyle \mathrm{div}\Big( \tilde a\Big( \frac{x}{\eps}, \xi_{\frac{t}{\eps^2}} \Big) u^\eps \Big)
\\[4mm]
u^\eps (0,x) & = & g(x)
\end{array} \right.
\end{equation}
with a diffusion process $\xi_s$,  $s\in(-\infty,+\infty)$,  with values in $\mathbb R^N$ or on a compact manifold.  This process is defined on a probability space $(\Omega, \mathcal{F},\mathbb{P})$. For the sake of definiteness
we consider here the case of a diffusion in  $\mathbb R^N$.    The corresponding It\^o equation
reads
$$
d\xi_t = b(\xi_t) dt +\sigma(\xi_t) dW_t,
$$
here $W_\cdot$ stands for a standard $N$-dimensional Wiener process. The infinitesimal generator of $\xi$ is denoted by
$\cal L$:
$$
\cL f(y)= q^{ij}(y) \frac{\partial^2}{\partial y^i\partial y^j}f(y) + b(y)\cdot\nabla f(y), \qquad y \in \R^N,
$$
with a $N\times N$ matrix $q(y)=\frac{1}{2} \sigma(y)\sigma^*(y)$.
We also introduce an operator
$$
\cA f(x) = \mathrm{div}_x \left( a\left( x, y \right) \nabla_x  f \right);
$$
here $y$ is a parameter.
%$\cA$ denotes $\cA^\eps$ for $\eps=1$.
Applied to a function $f(z,y)$, $\cL$ acts on the function $y \mapsto f(z,y)$ for $z$ fixed, and $\cA$ acts on the function $z \mapsto f(z,y)$ for $y$ fixed.

%As in \cite{KP_1995}, the following conditions on the %coefficients of $\cA^\eps$ and $\cL$ are supposed to hold.

%The coefficients in (\ref{eq:u^eps}) are defined by
%$$
%a(z,s,\omega)={\bf a}(z,{\rm T}_s\omega),
%$$
%where ${\bf a}(z,\omega)$ is a random symmetric uniformly positive definite matrix.

In the diffusive case condition  {\bf H3} can be replaced with certain assumptions on the generator $\cL$.
More precisely,
we suppose that the following conditions hold true.

\begin{enumerate}
\item[{\bf A1}.] The coefficients $a$ and $q$ are uniformly bounded as well as their first order derivatives in all variables:
$$|a(z,y)| + |\nabla_z a(z,y)| + |\nabla_y a(z,y)| \leq C_1,$$
$$|q(y)| + |\nabla q(y)| \leq C_1.$$
The function $b$ as well as its derivatives satisfy polynomial growth condition:
$$|b(y)| + |\nabla b(y)| \leq C_1(1+|y|)^{N_1}.$$

\item[{\bf A2}.] Both $\cA$ and $\cL$ are uniformly elliptic:
$$
C_2 \mathbf{I}\leq a(z,y) , \quad C_2\mathbf{I} \leq q(y) , \quad \mbox{with } C_2 > 0,
$$
where $\mathbf{I}$ stands for a unit matrix of the corresponding dimension.
\item[{\bf A3}.] There exist $N_2 > -1$, $R > 0$ and $C_3 > 0$ such that
$$b(y) \frac{y}{|y|} \leq -C_3|y|^{N_2}$$
for all $y$, $|y| > R$.
\end{enumerate}
Under above assumptions the process $\xi$ has a unique invariant probability measure (see \cite{PV}). This measure possesses a smooth density $\pi$ that forms the kernel of the formal adjoint operator $\cL^*$ of $\cL$.  We assume that $\xi_s$ is stationary. Then
$$
\E{ f(z,\xi_s) } = \int_{\R^N} f(z,y) \pi(y) dy
$$
 
\begin{remark}
Notice that conditions {\bf A1}--{\bf A3} need not imply condition ${\bf H3}$.
In general, mixing properties that follow from {\bf A1}--{\bf A3} are weaker than those stated by ${\bf H3}$. However, in the diffusive case  these conditions are sufficient
for the CLT type results used in the proofs below. This makes the diffusive case interesting. It should also be noted that
in this case the conditions  are given in terms of the process generator, which might be more comfortable in applications.
\end{remark}

\noindent
Let us recall the result of
\cite{KP97} (see also
\cite{CKP01}).
\begin{thm} \label{thm:KP_95}
Under Assumptions {\bf A1}--{\bf A3}, the solution $u^\eps$ of \eqref{eq:u^eps_bis} converges almost surely
in the space $L^2((0,T)\times\mathbb R^ n)$ to the solution of problem (\ref{hom_cauchy})
%\begin{equation} \label{eq:u^0}
%\left\{ \begin{array}{lcl}
%\ds  \partial_t u_0& = &\displaystyle \mathrm{div}\big(a^{\rm eff} \nabla u_0\big)\\[2.5mm]
%u_0 (0,x) & = & g(x)
%\end{array} \right.
%\end{equation}
with
%\begin{itemize}
%\item for $\al=2$, $\widehat a$ is given by
\begin{equation} \label{eq:a^al<2}
a^{\rm eff} = \int_{\mathbb R^N}\int_{\mathbb T^n} a({\bf I}+\nabla_z\chi^0)\pi(y)\, dzdy
\end{equation}
and $\chi^0$ being the solution of the following equation
\begin{equation} \label{chi0_al=2}
(\cA +  \cL) \chi^0  = - \mathrm{div}_z{a}(z,y).
\end{equation}
%\item for $0 < \al < 2$, the formula for $\widehat a$ remains unchanged, and $\chi^0$ satisfies:
%\begin{equation} \label{chi0_al<2}
%\cA \chi^0  = -\mathrm{div}_z a(z,y);
%\end{equation}
%\item for $\al > 2$,
%\begin{equation} \label{eq:a^al>2}
%\widehat a = \langle \overline{a}(1+\nabla_z \chi^0) \rangle,
%\end{equation}
%and $\chi^0$ is a solution of
%\begin{equation} \label{chi0_al>2}
%\overline{ \cA} \chi^0 = \mathrm{div} \left( \overline{a}\left(z \right) \nabla\chi^0 \right)  = - \mathrm{div} \overline a (z).
%\end{equation}
%\end{itemize}
\end{thm}

\section{Technical statements}\label{s_tech}

In this section we provide a number of technical statements required
for formulating and proving the main results.

Consider an equation
\begin{equation}\label{aux_para}
\partial_s\psi(z,s)-\mathrm{div}\big(a(z,s)\nabla\psi(z,s)\big)=\phi(z,s)
\end{equation}
with a stationary in $s$ and periodic in $z$ random function $\phi$.
\begin{lmm}\label{l_exist_stat}
Let $\phi\in L^2_{\rm loc}(\mathbb R;H^{-1}(\mathbb T^n))$, and assume that $\|\phi\|^2_{L^2((0,1);H^{-1}(\mathbb T^n))}\leq C$ with a non-random constant $C$. Assume, moreover, that
\begin{equation}\label{centre_rhs}
\int_{\mathbb T^n}\phi(z,s)dz=0\qquad \hbox{a.s.}
\end{equation}
Then equation (\ref{aux_para}) has a stationary solution $\psi\in L_{\rm loc}^\infty((-\infty,+\infty);L^2(\mathbb T^n))\cap
L_{\rm loc}^2((-\infty,+\infty);H^1(\mathbb T^n))$. It is unique up to an additive (random) constant,
and
\begin{equation}\label{aux_bounn}
\|\psi\|^2_{L^\infty(\mathbb R;L^{2}(\mathbb T^n))}\leq C_1,\qquad
\|\psi\|^2_{L^2((0,1);H^{1}(\mathbb T^n))}\leq C_1.
\end{equation}
\end{lmm}
\begin{proof}
Since a proof of this statement is similar to that of Lemmata 2 and 4  in \cite{KP02}, we provide here only
a sketch of the proof. Consider the Green function of (\ref{aux_para}). It solves  a Cauchy problem
$$
\begin{array}{c}
\displaystyle
\partial_s\mathcal{G}(z,z_0,s,s_0)-\mathrm{div}\big(a(z,s)\nabla\mathcal{G}(z,z_0,s,s_0)\big)=0,\qquad z\in\mathbb T^n,\ s\geq s_0,\\[2mm]
\displaystyle
\mathcal{G}(z,z_0,s_0,s_0)=\delta(z-z_0).
\end{array}
$$
From the Harnack inequality and maximum principle it easily follows (see \cite{KP02}) that for all $s\geq s_0+1$
\begin{equation}\label{green_est1}
%|\mathcal{G}(z,z_0,s,s_0)-1|\leq Ce^{-\nu(s-s_0)},\quad
\|\mathcal{G}(\cdot,z_0,\cdot,s_0)-1\|_{L^2((s,s+1);H^1(\mathbb T^n))}\leq Ce^{-\nu(s-s_0)}
\end{equation}
with deterministic  constants $C$ and $\nu>0$. Then we have
$$
\psi(z,s)=\int\limits_{-\infty}^s\int_{\mathbb T^n}\mathcal{G}(z,\hat z,s,\hat s)\phi(\hat z,\hat s)\,d\hat zd\hat s=
$$
$$
=\int\limits_{-\infty}^{s-1}\int\limits_{\mathbb T^n}\big(\mathcal{G}(z,\hat z,s,\hat s)-1\big)\phi(\hat z,\hat s)\,d\hat zd\hat s
+\int\limits^{s}_{s-1}\int\limits_{\mathbb T^n}\big(\mathcal{G}(z,\hat z,s,\hat s)\big)\phi(\hat z,\hat s)\,d\hat zd\hat s,
$$
here we have also used (\ref{centre_rhs}).   The first term on the right-hand side  can be estimated with the help of
(\ref{green_est1}), the second one by means of the standard energy inequality. This yields the first bound in   (\ref{aux_bounn}).
 By construction, $\psi(z,s)$ is a stationary solution of (\ref{aux_para}).
The second bound  in   (\ref{aux_bounn}) readily follows from the first one.
\end{proof}

\begin{corollary}\label{cor_t1}
If the function $\phi$ in (\ref{aux_para}) belongs to $L^\infty(\mathbb R;W^{-1,\infty}(\mathbb T^n))$, then
$\psi\in L^\infty(\mathbb R\times \mathbb T^n)$ and
$$
\|\psi\|_{L^\infty(\mathbb R\times \mathbb T^n)}\le C \|\phi\|_{L^\infty(\mathbb R;W^{-1,\infty}(\mathbb T^n))}
$$
with a deterministic constant $C$.
\end{corollary}
\begin{proof}
This statement follows from Lemma \ref{l_exist_stat} due to the Nash type estimates  for solutions of parabolic equations
(see \cite[Theorem VII,3.1]{GT}).
\end{proof}

\bigskip
Denote by $\mathcal{F}^{a,\phi}_{\leq T}$  the $\sigma$-algebra  $\sigma\{a(z,s),\phi(x,s)\,:\,s\leq T\}$.  The $\sigma$-algebra
$\mathcal{F}^{a,\phi}_{\geq T}$ is defined accordingly. Let $\rho_{a,\phi}(r)$ be  maximum correlation coefficient of $(a,\phi)$.
Denote also
$$
l(s)=\int_{\mathbb T^n}\big(a(z,s)\nabla_z\psi(z,s)-\E(a(z,s)\nabla_z\psi(z,s))\big)\,dz.
$$
\begin{lmm}\label{l_mixi}
For the vector-function $l(\cdot)$ the following estimate holds
$$
\|\E\{l(s)\,|\, \mathcal{F}^{a,\phi}_{\leq 0}\}\|\big._{L^2(\Omega)}\leq C\big(e^{-\nu s/2}+\rho_{a,\phi}(s/2)\big), \qquad
\nu>0.
$$
\end{lmm}
\begin{proof}
This inequality has been proved in \cite[Proof of Lemma 3]{KP02}.  Here we provide an outline of the proof. We represent
$$
\psi(z,s)=\psi^1(z,s)+\psi^2(z,s)=
$$
$$
=\int\limits_{-\infty}^{s/2}\int\limits_{\mathbb T^n}\big(\mathcal{G}(z,\hat z,s,\hat s)-1\big)\phi(\hat z,\hat s)\,d\hat zd\hat s
+\int\limits^{s}_{s/2}\int\limits_{\mathbb T^n}\big(\mathcal{G}(z,\hat z,s,\hat s)\big)\phi(\hat z,\hat s)\,d\hat zd\hat s.
$$
Then
$$
l(s)=l^1(s)+l^2(s),\quad l^{1,2}(s)=\!\int\limits_{\mathbb T^n}\!\big(a(z,s)\nabla_z\psi^{1,2}(z,s)-\E(a(z,s)\nabla_z\psi^{1,2}(z,s))\big)\,dz.
$$
Considering (\ref{green_est1}) we get  $\|l^1(s)\|_{L^2(\Omega)}\leq Ce^{-\nu s/2}$.  Since $l^2(s)$ is $ \mathcal{F}^{a,\phi}_{\geq s/2}$-measurable,
we obtain $\|\E\{l^2(s)\,|\, \mathcal{F}^{a,\phi}_{\leq 0}\}\|\big._{L^2(\Omega)}\leq C\rho_{a,\phi}(s/2)$. This yields the desired inequality.
\end{proof}

%If, in addition,
%$$
%\int_0^\infty (\rho_{a,g}(r))^\beta dr<+\infty
%$$
%for some $\beta>0$, then
%$$
%\int_0^\infty (\rho_{a,g,\psi}(r))^\beta dr<+\infty.
%$$
%If $\rho_{a,g}(r)\le C(1+r)^{-\beta}$, then $\rho_{a,g,\psi}(r)\le C_1(1+r)^{-\beta}$.

\section{Formal asymptotic expansion}\label{s_asy}

In this section we deal with the formal asymptotic expansion of a solution of problem (\ref{ory_cauchy}).
Although, in contrast with the periodic case,  this method fails to work in full generality  in the case under consideration,
we can use it in order to understand the structure of the leading terms of the difference $u^\eps-u^0$. As usually in the multi-scale asymptotic expansion method we consider $z=x/\eps$ and $s=t/\eps^2$ as independent variables
and use repeatedly the formulae
$$
\begin{array}{c}
%\displaystyle
\frac{\partial}{\partial x_j}f\Big(x,\frac{x}{\eps}\Big)=\Big(\frac{\partial}{\partial x_j}f(x,z)+
\frac{1}{\eps}\frac{\partial}{\partial z_j}f(x,z)\Big)_{z=\frac{x}{\eps}},\\[4mm]
%\displaystyle
\frac{\partial}{\partial t}f\Big(t,\frac{t}{\eps^2}\Big)=\Big(\frac{\partial}{\partial t}f(t,s)+
\frac{1}{\eps^2}\frac{\partial}{\partial s}f(t,s)\Big)_{s=\frac{t}{\eps}}.
\end{array}
$$
We
%begin with the case $\alpha=2$ and
represent a solution $u^\eps$ as the following
asymptotic series in integer powers of $\eps$:
\begin{equation}\label{asy}
u^\eps(x,t)=u^0(x,t)+\eps u^1\Big(x,t,\frac{x}{\eps},\frac{t}{\eps^2}\Big)+\eps^2 u^2\Big(x,t,\frac{x}{\eps},\frac{t}{\eps^2}\Big)+\dots\,;
\end{equation}
here all the functions $u^j(x,t,z,s)$ are periodic in $z$.
The dependence in $s$ is not always stationary.
% We assume that for all $j\geq 1$

Substituting the expression on the right-hand side of (\ref{asy}) for $u^\eps$ in (\ref{eq:u^eps}) and collecting
power-like terms in (\ref{eq:u^eps}) yields
\begin{itemize}
\item[$(\eps^{-1}):$]\hskip 1.2 cm$\partial_s
u^1-\mathrm{div}_z\big(a(z,s)\nabla_zu^1\big)=
-\mathrm{div}_z\big(a(z,s) \nabla_x u^0\big)$.
\item[$(\eps^{0})\ \,:$]\hskip 1 cm$
\begin{array}{r}\partial_s
u^2-\mathrm{div}_z\big(a(z,s)\nabla_zu^2\big)=
-\partial_t u^0
+\mathrm{div}_x\big(a(z,s) \nabla_x u^0\big)\ \\[1.3mm]
+\mathrm{div}_z\big(a(z,s) \nabla_x u^1\big)
+\mathrm{div}_x\big(a(z,s) \nabla_z u^1\big).
\end{array}$
\item[$(\eps^{1})\ \,:$]\hskip 1 cm
$
\begin{array}{r}\partial_s
u^3-\mathrm{div}_z\big(a(z,s)\nabla_zu^3\big)=
-\partial_t u^1
+\mathrm{div}_x\big(a(z,s) \nabla_x u^1\big)\ \\[1.3mm]
+\mathrm{div}_z\big(a(z,s) \nabla_x u^2\big)
+\mathrm{div}_x\big(a(z,s) \nabla_z u^2\big).
\end{array}$
\end{itemize}
We will see later on that dealing with the first three equations is sufficient.

\medskip
In equation ($\eps^{-1}$) the variables $x$ and $t$ are parameters.
By Lemma \ref{l_exist_stat} this equation  has the unique stationary solution.
The fact that the right-hand side of the equation is of the form $[\mathrm{div}_z\big(a(z,s)]\cdot\nabla_x u^0$ suggests that
$$
u^1(x,t,z,s)=\chi(z,s)\nabla u^0(x,t).
$$
with a vector-function $\chi=\{\chi^j(z,s)\}_{j=1}^n$ solving equation \eqref{aux_alp2} that reads
$$
\partial_s
\chi-\mathrm{div}_z\big(a(z,s)\nabla_z\chi\big)=
\mathrm{div}_z\big(a(z,s)\big),
$$
$\mathrm{div}_z a(z,s)$ stands for $\frac{\partial}{\partial z^i}a^{ij}(z,s)$.
By Lemma \ref{l_exist_stat} and Corollary \ref{cor_t1} we have
$\chi\in (L^\infty(\mathbb R\times \mathbb T^n))^n\cap (L^2_{\rm loc}(\mathbb R; H^1(\mathbb T^n)))^n$,
and
\begin{equation}\label{chi_est}
\|\chi^j\|_{L^\infty(\mathbb R\times \mathbb T^n)}\le C,\qquad
\|\chi^j\|_{L^2([0,1];H^1(\mathbb T^n))}\le C,\quad j=1,\ldots,n.
\end{equation}
with a deterministic constant $C$.
For the sake of definiteness we assume from now on that
\begin{equation}\label{norm_coo}
\int_{\mathbb T^n} \chi(z,s)dz=0.
\end{equation}
One can easily check that this integral does not depend on $s$ so that the normalization condition makes sense.
%The first inequality follows from the Nash elliptic estimates, see \cite{GT}, to obtain the second one it suffices to multiply the equation related to
%$\eps^{-1}$ by $\chi$ and integrate the resulting relation over $\mathbb T^n\times[0,1]$.

We turn to the terms of order $\eps^0$. We do not reprove here the homogenization results (see \cite{JKO82}) and assume that $u^0$ satisfies problem (\ref{hom_cauchy}) with $a^{\rm eff}$ given by (\ref{aeff-ravno}). Then the right-hand side of  equation ($\eps^0$) takes the form
$$
-\partial_t u^0
+\mathrm{div}_x\big(a(z,s) \nabla_x u^0\big)
+\mathrm{div}_z\big(a(z,s) \nabla_x u^1\big)
+\mathrm{div}_x\big(a(z,s) \nabla_z u^1\big)=
$$
$$
=
\mathrm{div}_x\big(\{a(z,s)(\mathbf{I}+\nabla_z\chi(z,s)) -a^{\rm eff}\}\nabla_x u^0\big)
+\mathrm{div}_z\big(a(z,s) \nabla_x u^1\big)  %:=\Psi_0.
$$
By the definition of $a^{\rm eff}$  (see (\ref{aeff-ravno})) we have
$$
\E\int_{\mathbb T^n}\{a(z,s)(\mathbf{I}+\nabla_z\chi(z,s)) -a^{\rm eff}\}\,dz=0.
$$
Letting
\begin{equation}\label{defpsi21}
\Psi_{2,1}(s)=\int\limits_{\mathbb T^n}\{a(z,s)(\mathbf{I}+\nabla_z\chi(z,s)) -a^{\rm eff}\}\,dz
\end{equation}
and
\begin{equation}\label{defpsi22}
\Psi_{2,2}(z,s)\!=\!\{a(z,s)(\mathbf{I}+\nabla_z\chi(z,s))\! -\!a^{\rm eff}\}\!-\!\Psi_{2,1}(s)+\mathrm{div}_z \big(\!a(z,s)\otimes\chi(z,s)\!\big)
\end{equation}
with
$$
\mathrm{div}_z \big(a(z,s)\otimes\chi(z,s)=\big\{\frac{\partial}{\partial z^i}a^{ij}(z,s)\chi^k(z,s)\big\}_{j,k=1}^n,
$$
we rewrite equation  ($\eps^{0}$) as follows
\begin{equation}\label{mod_eps0}
\partial_s
u^2-\mathrm{div}_z\big(a(z,s)\nabla_zu^2\big)=
\big(\Psi^{ij}_{2,1}(s)+\Psi^{ij}_{2,2}(z,s)\big)\frac{\partial^2}{\partial x^i\partial x^j}u^0.
\end{equation}

Since the process $\int_0^s\Psi_{2,1}(r)dr$ need not be stationary, we cannot follow any more the same strategy as in the periodic case.  Instead,  we consider the equation
\begin{equation} \label{eq:aux004}
\left\{ \begin{array}{lcl} \displaystyle
\frac{\partial V^{\eps,1}}{\partial t}& = &\displaystyle \mathrm{div}\Big( a\Big( \frac{x}{\eps}, \frac{t}{\eps^2} \Big)\nabla V^{\eps,1}\Big)+\Psi^{ij}_{2,1}\Big(\frac{t}{\eps^2}\Big)
\frac{\partial^2}{\partial x^i\partial x^j}u^0(x,t)   \\[4mm]
V^{\eps,1} (0,x) & = & 0.
\end{array} \right.
\end{equation}
This suggests the representation
\begin{equation}\label{asy_modif}
\begin{array}{rl}
\displaystyle
u^\eps(x,t)
= \!\!&\!\! \displaystyle
u^0(x,t)+\eps \chi\Big(\frac{x}{\eps},\frac{t}{\eps^2}\Big)\nabla u^0(x,t)
\\[2mm]
&\!\!  \displaystyle
+V^{\eps,1}+\eps^2 v^2\Big(x,t,\frac{x}{\eps},\frac{t}{\eps^2}\Big)+\dots
\end{array}
\end{equation}
with
$$
v^2(x,t,z,s)=\chi_{2,2}^{ij}(z,s)\frac{\partial^2}{\partial x^i\partial x^j}u^0(x,t)\,,
$$
where
$\chi_{2,2}^{ij}(z,s)$ is a stationary zero mean solution of the equation
\begin{equation}\label{eq_chi22}
\partial_s\chi_{2,2}^{ij}(z,s)-\mathrm{div}_z\big(a(z,s)
\nabla_z\chi_{2,2}^{ij}(z,s)\big)=\Psi_{2,2}^{ij}(z,s).
\end{equation}
It is straightforward to check that due to (\ref{chi_est})--(\ref{defpsi22}) we have
$$
%\|\Psi\big._{2,1}^{ij}\|_{L^2(0,1)}\leq C,\qquad
 \|\Psi\big._{2,2}^{ij}\|_{L^2((0,1);H^{-1}(\mathbb T^n))}\leq C, \qquad i,j=1,\ldots,n.
$$
Then the conditions of Lemma \ref{l_exist_stat} are fulfilled for  equation (\ref{eq_chi22}) and, therefore,
this equation has a stationary solution that satisfies the estimate
\begin{equation}\label{chi02_est}
\|\chi\big._{2,2}^{ij}\|_{L^2([0,1];H^1(\mathbb T^n))}+
\|\chi\big._{2,2}^{ij}\|_{L^\infty((-\infty.+\infty);L^2(\mathbb T^n))}\le C, \qquad i,j=1,\ldots,n.
\end{equation}
with a deterministic constant $C$.

By its definition,
$\Psi_{2,1}(s)$ is a stationary zero average process. Denote
$$
\chi_{2,1}^{ij}(s)=\int_0^s\Psi^{ij}_{2,1}(r)dr.
$$
Estimates (\ref{chi_est}) imply that
$$
\|\Psi^{ij}_{2,1}\|_{L^2(0,1)}\le C,\qquad i,j=1,\ldots,n.
$$
with a deterministic constant $C$.
It follows from Lemmata \ref{l_exist_stat} and \ref{l_mixi} that under condition
{\bf H3}  it holds
$$
\int_0^\infty\|\E\{\Psi_{2,1}(s)\,|\,\mathcal{F}_{\leq 0}^{\Psi_{2,1}}\}\|\big._{(L^2(\Omega))^{\,n^2}}\,ds\leq
C\int_0^\infty\big(e^{-\nu s/2}+\rho_{\Psi_{2,1}}(s/2)\big)dy<\infty.
$$
Therefore, the invariance principle holds for this process (see \cite[Theorem VIII.3.79]{JSh}),
that is for any $T>0$
\begin{equation}\label{inv_p}
\eps\chi_{2,1}\Big(\frac{\cdot}{\eps^2}\Big) \ \mathop{\longrightarrow}\limits_{\eps\to0}\
\Lambda^{1/2} W_\cdot
\end{equation}
in law in the space $(C[0,T])\big.^{n^2}$ with
$$
\Lambda^{ijkl}=\int\limits_0^\infty\E\big(\Psi_{2,1}^{ij}(0) \Psi_{2,1}^{kl}(s)
+\Psi_{2,1}^{kl}(0)\Psi_{2,1}^{ij}(s)\big)\,ds,
$$
here $W$ is a standard $n^2$-dimensional Wiener process.
Since the $n^2\times n^2$ matrix $\Lambda$ is symmetric and positive (but not necessary positive definite), its square root is well defined.

\begin{remark}
One can see that the processes $\chi_{2,1}$ and $\chi_{2,2}$ show rather different behaviour. In fact, since  the process $\chi_{2,2}$ is stationary,
the function $\eps\chi_{2,2}(x/\eps,t/\eps^2)$ goes to zero, as $\eps\to0$. To the contrary, by the Cental Limit Theorem type arguments, the process
$\eps \chi_{2,1}(t/\eps^2)$ need not tend to zero on $[0,T]$, and, thus, it contributes to the asymptotics in question. Under our standing conditions,
this process is of order one.
\end{remark}
\begin{lemma}\label{l_stochpde}
The functions $\eps^{-1} V^{\eps,1}$ converges in law, as $\eps\to0$, in the space $C((0,T);L^2(\mathbb R^n))$ to the unique solution of the following SPDE with a finite dimensional additive noise:
\begin{equation}\label{aux_spde}
\left\{ \begin{array}{lcl} \displaystyle
dV^{0,1}& = &\displaystyle \mathrm{div}( a^{\rm eff}\nabla V^{0,1})dt+\big(\Lambda^{1/2}\big)^{ijkl}
\frac{\partial^2}{\partial x^i\partial x^j}u^0(x,t)dW_{t,kl}  \\[4mm]
V^{0,1} (0,x) & = & 0.
\end{array} \right.
\end{equation}
\end{lemma}
\begin{proof}
The proof is a consequence of (\ref{inv_p}) and the fact that $u^0(x,t)$ is a smooth deterministic function
vanishing with its derivatives at infinity. To see this we introduce  an auxiliary function $\check V^\eps$ as the solution
to the following Cauchy problem
%is standard, see \cite{??} for the details.
$$
%\begin{equation} \label{eq:aux004}
\left\{ \begin{array}{lcl} \displaystyle
\frac{\partial \check V^{\eps}}{\partial t}& = &\displaystyle \mathrm{div}( a^{\rm eff} \nabla \check V^{\eps})+\frac{1}{\eps}\Psi^{ij}_{2,1}\Big(\frac{t}{\eps^2}\Big)
\frac{\partial^2}{\partial x^i\partial x^j}u^0(x,t)   \\[4mm]
\check V^{\eps} (0,x) & = & 0.
\end{array} \right.
%\end{equation}
$$
For the sake of brevity we denote $v^0_{ij}(x,t)=\frac{\partial^2}{\partial x^i\partial x^j}u^0(x,t) $. Notice that $v^0_{ij}$ solves
the equation $\partial_t v^0_{ij}-\mathrm{div}(a^{\rm eff}\nabla v^0_{ij})\Big.$ for all $i,j=1,\ldots,n$. Then one can easily check that
\begin{equation}\label{explicit}
\check V^{\eps}(x,t)=\eps\chi^{ij}_{2,1}\Big(\frac{t}{\eps}^2\Big)v^0_{ij}(x,t)
\end{equation}
Our first goal is to show that
\begin{equation}\label{o_to_e}
\|\eps^{-1}V^{\eps,1}-\check V^\eps\|_{L^2((0,T)\times\mathbb R^n)}\,\longrightarrow 0 \qquad \hbox{in probability.}
\end{equation}
To this end we represent $\eps^{-1}V^{\eps,1}$ as
$$
\eps^{-1}V^{\eps,1}=\eps\chi^{ij}_{2,1}\Big(\frac{t}{\eps}^2\Big)v^0_{ij}(x,t)+Z^\eps(x,t)
$$
and substitute it in (\ref{eq:aux004}). This yields the following equation for $Z^\eps$\,:
$$
%\begin{equation} \label{eq:aux004}
\left\{ \begin{array}{lcl} \displaystyle
\frac{\partial  Z^{\eps}}{\partial t}= \displaystyle \mathrm{div}\Big( a\Big(\frac{x}{\eps},\frac{t}{\eps^2}\Big) \nabla Z^{\eps}\Big)+
\eps\chi^{ij}_{2,1}\Big(\frac{t}{\eps^2}\Big)
\Big\{\partial_t v^0_{ij}-\mathrm{div}\Big(a\Big(\frac{x}{\eps},\frac{t}{\eps^2}\Big)\nabla v^0_{ij}\Big)\Big\}
 \\[4mm]
Z^{\eps} (0,x)  =  0.
\end{array} \right.
%\end{equation}
$$

Let $\rho=\rho(t)$ be a continuous function on $[0,T]$. Then
\begin{equation}\label{paraene}
\Big\|\rho(t)\Big\{\partial_t v^0_{ij}-\mathrm{div}\Big(a\Big(\frac{x}{\eps},\frac{t}{\eps^2}\Big)\nabla v^0_{ij}\Big)\Big\}\Big\|_{L^2(0,T; H^{-1}(\mathbb R^n))}\leq C\|\rho\|_{L^\infty(0,T)},
\end{equation}
where the constant $C$ does not depend on $\eps$.   Next, we consider the following Cauchy problem:

\begin{equation} \label{e_call_z}
\left\{\! \begin{array}{lcl} \displaystyle
\frac{\partial  \mathcal{Z}^{\eps}}{\partial t}= \displaystyle \mathrm{div}\Big( a\Big(\frac{x}{\eps},\frac{t}{\eps^2}\Big) \nabla \mathcal{Z}^{\eps}\Big)+
\rho(t)
\Big\{\partial_t v^0_{ij}-\mathrm{div}\Big(a\Big(\frac{x}{\eps},\frac{t}{\eps^2}\Big)\nabla v^0_{ij}\Big)\Big\}
 \\[4mm]
\mathcal{Z}^{\eps} (0,x)  =  0.
\end{array} \right.
\end{equation}
With the help of energy estimates we derive from  (\ref{paraene}) that
$$
\| \mathcal{Z}^\eps\|_{L^2(0,T;H^1(\mathbb R^n))}+\| \partial_t \mathcal{Z}^\eps\|_{L^2(0,T;H^{-1}(\mathbb R^n))}\leq C\|\rho\|_{L^\infty(0,T)}.
$$
Taking into account the fast decay of $v^0$ and its derivatives at infinity we deduce from this estimate (see \cite{Lio}) that almost surely
for a subsequence the function $\mathcal{Z}^\eps$ converges in $C([0,T];L^2(\mathbb R^n))$ to some function  $\mathcal{Z}^0$.
In order to characterize $\mathcal{Z}^0$, assume for a while that $\rho$ is smooth. For an arbitrary $\varphi\in C_0^\infty((0,T)\times\mathbb R^n)$
we use in the integral identity of problem (\ref{e_call_z})  the following test function
$$
\varphi^\eps(x,t)=\varphi(x,t)+\eps\chi\big._-\Big(\frac{x}{\eps},\frac{t}{\eps^2}\Big)\nabla\varphi(x,t)
$$
with $\chi\big._-$ defined in (\ref{inv_aux}).
Setting $a^\eps(x,t)=a\big(\frac{x}{\eps},\frac{t}{\eps^2}\big)$, $\chi_-^\eps(x,t)=\chi\big._-\big(\frac{x}{\eps},\frac{t}{\eps^2}\big)$, after integration by parts
in this integral identity and straightforward rearrangements we obtain
$$
-\int\limits_0^T\!\!\int\limits_{\mathbb R^n}\!\mathcal{Z}^\eps\big(\partial_t\varphi +  (a^\eps)^{ij}\partial_{x^i}\partial_{x^j} \varphi
+ (a^\eps)^{ij}\partial_{z^j}(\chi_-^\eps)^k\partial_{x^i}\partial_{x^k}\varphi+ \partial_{z^i}[(a^\eps)^{ij}(\chi_-^\eps)^k]\partial_{x^j}\partial_{x^k}\varphi\big)\,dxdt
$$
$$
-\eps^{-1}\int\limits_0^T\!\!\int\limits_{\mathbb R^n}\!\mathcal{Z}^\eps\big( \partial_{z^i}(a^\eps)^{ij}\partial_{x^j}\varphi+\partial_s(\chi_-^\eps)^j
\partial_{x^j}\varphi + \partial_{z^i}[(a^\eps)^{ij}\partial_{z^j}(\chi_-^\eps)^k] \partial_{x^k}\varphi \big)\,dxdt
$$
$$
-\eps\int\limits_0^T\!\!\int\limits_{\mathbb R^n}\!\mathcal{Z}^\eps\big((a^\eps)^{ij}(\chi_-^\eps)^k\partial_{x^i}\partial_{x^j}\partial_{x^k}\varphi
+(\chi_-^\eps)^j \partial_t\partial_{x^j}\varphi \big)\,dxdt
$$
$$
=\int\limits_0^T\!\!\int\limits_{\mathbb R^n}\!\big(\rho\varphi\partial_t v^0_{lm} - \rho v^0_{lm}\{(a^\eps)^{ij}\partial_{x^i}\partial_{x^j}\varphi
-(a^\eps)^{ij}\partial_{z^i}(\chi_-^\eps)^k \partial_{x^j} \partial_{x^k}\varphi-\partial_{z^j}[(a^\eps)^{ij}(\chi_-^\eps)^k] \partial_{x^i} \partial_{x^k}\varphi\}\big)\,dxdt
$$
$$
-\eps^{-1}\int\limits_0^T\!\!\int\limits_{\mathbb R^n}\!\big( \rho v^0_{lm}\{\partial_s(\chi_-^\eps)^k+\partial_{z^i}((a^\eps)^{ij}\partial_{z^j}(\chi_-^\eps)^k)+\partial_{z^i}(a^\eps)^{ik}\}\partial_{x^k}\varphi\big)\,dxdt
$$
$$
-\eps\int\limits_0^T\!\!\int\limits_{\mathbb R^n}\!\big( v^0_{lm}(\chi_-^\eps)^k\partial_t(\rho\varphi)+ \rho v^0_{lm}(a^\eps)^{ij}(\chi_-^\eps)^k
 \partial_{x^i}  \partial_{x^j} \partial_{x^k}\varphi\big)\,dxdt
$$
Notice that due to equation (\ref{inv_aux}) all the terms of order $\eps^{-1}$ are equal to zero. Passing to the limit, as $\eps\to 0$ yields
$$
\int\limits_0^T\!\!\int\limits_{\mathbb R^n}\!\mathcal{Z}^0(\partial_t\varphi +\mathrm{div}(a^{\rm eff}\nabla\varphi))\,dxdt=
\int\limits_0^T\!\!\int\limits_{\mathbb R^n}\!\big(\rho\varphi\partial_t v^0_{lm} - \rho v^0_{lm}\mathrm{div}(a^{\rm eff}\nabla\varphi)\big)\,dxdt
$$
Since $v^0_{lm}$ solves the effective equation, the integral on the right-hand side is equal to zero. Therefore,
$$
\partial_t\mathcal{Z}^0-\mathrm{div}(a^{\rm eff}\nabla\mathcal{Z}^0)=0.
$$
Since $\mathcal{Z}^0(x,0)=0$, we conclude that  $\mathcal{Z}^0=0$.\\
By the density arguments, $\mathcal{Z}^0=0$ for any continuous $\rho$.  Due to the tightness of the family $\big\{\eps\chi^{ij}_{2,1}\big(\frac{t}{\eps^2}\big)\big\}$
in $C[0,T]$ this implies that $Z^\eps$ converges to zero in probability in $L^2(\mathbb R^n\times(0,T))$, and (\ref{o_to_e}) follows.

It remains to  pass to the limit in (\ref{explicit}) and check that the limit process satisfies (\ref{aux_spde}). Due to (\ref{inv_p}) and (\ref{explicit}),
 $\check V^{\eps}$ converges in law in $C(0,T;L^2(\mathbb R^n))$ to the process $\Lambda^{1/2}W_\cdot v^0$ with $n^2$-dimensional Wiener process
 $W_t$. Recalling the definition of $v^0_{ij}$, we obtain the desired convergence.
\end{proof}

\bigskip
We proceed with equation $(\eps^1)$. Its right-hand side can be rearranged as follows:
$$
-\partial_t u^1
+\mathrm{div}_x\big(a(z,s) \nabla_x u^1\big)\ \\[1.3mm]
+\mathrm{div}_z\big(a(z,s) \nabla_x v^2\big)
+\mathrm{div}_x\big(a(z,s) \nabla_z v^2\big)
$$
$$
=\big\{-a^{\rm eff}\otimes \chi(z,s)+a(z,s)\otimes \chi(z,s)
+\mathrm{div}_z[a(z,s)\otimes\chi_{2,2}(z,s)]
$$
$$
+a(z,s)\nabla_z\chi_{2,2}(z,s)
\big\}\frac{\partial^3}{\partial x^3}u^0(x,t)
%+\mathrm{div}_ya(y,s)\otimes\chi^{0,1}(s)\partial^3_{xxx}u^0(x,t)
:=\Psi_{3}(z,s)
%+\Psi^{1,1}(y,s)\big)
\frac{\partial^3}{\partial x^3}u^0(x,t);
$$
here and in what follows  the symbol $\frac{\partial^3}{\partial x^3}u^0(x,t)$ stands for the tensor of third order
partial derivatives of  $u^0$, that is  $\frac{\partial^3}{\partial x^3}=\big\{\frac{\partial^3}{\partial x^i \partial x^j\partial x^k}\big\}_{i,j,k=1}^n$; we have also denoted
$$
a(z,s)\otimes \chi(z,s)=\big\{a^{ij}(z,s)\chi^k(z,s)\big\}_{i,j,k=1}^n
$$
and
$$
 \mathrm{div}_z[a(z,s)\otimes\chi_{2,2}(z,s)]=
\big\{\partial_{z^i}[a^{ij}(z,s)\chi_{2,2}^{kl}(z,s)]\big\}_{j,k,l=1}^n.
$$
We introduce the following constant tensor $\mu=\{\mu^{ijk}\}_{i,j,k=1}^n$:
$$
\mu=\E\int_{\mathbb T^n}\big\{-a^{\rm eff}\otimes \chi(z,s)+a(z,s)\otimes \chi(z,s)+a(z,s)\nabla_z\chi\big._{2,2}(z,s)
\big\}dz
$$
with $a(z,s)\nabla_z\chi\big._{2,2}(z,s)=\{a^{ij}(z,s)\partial_{z^l}\chi^{lk}_{2,2}(z,s)\}_{i,j,k=1}^n$,
and consider the following problems:
\begin{equation}  \label{eq-aux101}
\left\{\! \begin{array}{lcl} \displaystyle
\frac{\partial\Xi_{\eps,1}}{\partial t}& \!\!\!=\!\!\! &\displaystyle \mathrm{div}\Big( a\Big( \frac{x}{\eps}, \frac{t}{\eps^2} \Big)\nabla \Xi_{\eps,1}\Big)+\Big(\Psi_{3}\big(\frac{x}{\eps},\frac{t}{\eps^2}\big)-\mu\Big)
\frac{\partial^3}{\partial x^3}u^0(x,t)   \\[4mm]
\Xi_{\eps,1} (x,0) & \!\!\!=\!\!\! & 0,
\end{array} \right.
\end{equation}
and
\begin{equation} \label{eq:aux102}
\left\{ \begin{array}{lcl} \displaystyle
\frac{\partial\Xi_{\eps,2}}{\partial t}& = &\displaystyle \mathrm{div}\Big( a\Big( \frac{x}{\eps}, \frac{t}{\eps^2} \Big)\nabla \Xi_{\eps,2}\Big)+\mu
\frac{\partial^3}{\partial x^3}u^0(x,t)   \\[4mm]
\Xi_{\eps,2} (0,x) & = & 0.
\end{array} \right.
\end{equation}
%and
%\begin{equation} \label{eq:aux103}
%\left\{ \begin{array}{lcl} \displaystyle
%\frac{\partial\Xi^{\eps,3}}{\partial t}& = &\displaystyle %\mathrm{div}\Big( a\Big( \frac{x}{\eps}, \frac{t}{\eps^2} %\Big)\nabla %\Xi^{\eps,3}\Big)+\Psi^{1,1}\big(\frac{x}{\eps},\frac{t}{\eps^2}\big)
%\frac{\partial^3}{\partial x^3}u^0(x,t)   \\[4mm]
%\Xi^{\eps,3} (0,x) & = & 0.
%\end{array} \right.
%\end{equation}
\begin{lemma}\label{l_small_h1}
The solution of problem (\ref{eq-aux101}) tends to zero a.s., as $\eps\to0$, in $L^2(\mathbb R^n\times[0,T])$. Moreover,
$$
\lim\limits_{\eps\to0}\E\big(\|\Xi_{\eps,1}\|^2_{L^2(\mathbb R^n\times[0,T])}\big)=0.
$$
\end{lemma}
\begin{proof}
Splitting further the term  $(\Psi_{3}-\mu)$ on the right-hand side of (\ref{eq-aux101}) into two parts
$$
\begin{array}{rl}
\displaystyle
\Psi_{3}(z,s)-\mu\!\!&\!\!=\mathrm{div}_z[a(z,s)\otimes\chi\big._{2,2}(z,s)]
\\[2mm] &\!\! \displaystyle
+\big\{(a(z,s)-a^{\rm eff})\otimes \chi(z,s)
+a(z,s)\nabla_z\chi\big._{2,2}(z,s)-\mu
\big\}
\\[2mm] &\!\! \displaystyle
=\Psi_{3,1}(z,s)+(\Psi_{3,2}(z,s)-\mu),
\end{array}
$$
we represent the solution $\Xi_{\eps,1}$ as the sum $\Xi^1_{\eps,1}$ and $\Xi^2_{\eps,1}$, respectively.

Since the right-hand side $g$ in (\ref{ory_cauchy}) satisfies condition {\bf H2}, the entries of $\frac{\partial^3}{\partial x^3}u^0$ are $C^1(\mathbb R^n)$ functions, and, moreover,  for any $K>0$ there exists $C_K$ such that
$$
\Big|\frac{\partial^3}{\partial x^3}u^0(x,t)\Big|+\Big|\frac{\partial^4}{\partial x^4}u^0(x,t)\Big|\leq C_K(1+|x|)^{-K} ,\qquad t\in[0,T].
$$
Combining this with (\ref{chi02_est}) we conclude that
$$
\Big\|\Psi_{3,1}\Big(\frac x\eps.\frac{t}{\eps^2}\Big)\frac{\partial^3}{\partial x^3}u_0(x,t)\Big\|_{L^2([0,T];H^{-1}(\mathbb R^n))}\le C\eps.
$$
Therefore,
\begin{equation}\label{xi_esti1}
\|\Xi^1_{\eps,1}\|_{L^2([0,T];H^1(\mathbb R^n))}\le C\eps.
\end{equation}
Due to (\ref{chi02_est}) and the properties of $u_0$, we have
$$
\big\|\Big(\Psi_{3,2}\Big(\frac{x}{\eps},\frac{t}{\eps^2}\Big)-\mu\Big)
\frac{\partial^3}{\partial x^3}u_0\Big\|_{L^2(\mathbb R^n\times(0,T))}\le C
$$
with a deterministic $C$. Using Theorem 1.5.1 in \cite{Lio} we derive from this estimate that a.s. the family $\Xi_{\eps,1}^2$ is compact in $L^2((0,T);L^2_{\mathrm{loc}}(\mathbb R^n))$.  Considering condition
{\bf H2} and Aronson's estimate (see \cite{Aro}), we then  conclude that the family $\Xi_{\eps,1}^2$
is compact in $L^2(\mathbb R^n\times(0,T))$.

By the Birkhoff ergodic theorem, the function $\big(\Psi_{3,2}\big(\frac{x}{\eps},\frac{t}{\eps^2}\big)-\mu\big)
\frac{\partial^3}{\partial x^3}u_0$  converges a.s. to zero weakly in $L^2(\mathbb R^n\times(0,T))$. Combining this with the above compactness arguments, we conclude that a.s. $\Xi_{\eps,1}^2$ converges to zero in $L^2(\mathbb R^n\times(0,T))$. Then in view of (\ref{xi_esti1}),  $\Xi_{\eps,1}$ tends to zero in $L^2(\mathbb R^n\times(0,T))$ a.s. This yields the first statement of the lemma. The second
statement follows from the first one by the Lebesgue dominated convergence theorem.
 \end{proof}

\bigskip
According to \cite{JKO82}, problem (\ref{eq:aux102}) admits homogenization. In particular, $\Xi_{\eps,2}$ converges a.s.
in $L^2(\mathbb R^n\times(0,T))$  to a solution of the following problem:
\begin{equation}\label{eq_homaux}
\left\{ \begin{array}{lcl} \displaystyle
\frac{\partial\Xi_{0,2}}{\partial t}& = &\displaystyle \mathrm{div}\big( a^{\rm eff}\nabla \Xi_{0,2}\big)+\mu
\frac{\partial^3}{\partial x^3}u^0(x,t)   \\[4mm]
\Xi_{0,2} (0,x) & = & 0
\end{array} \right.
\end{equation}

This is not the end of the story with the asymptotic expansion
because the initial condition is not satisfied at the level
$\eps^1$. In order to fix this problem we introduce one more term of order $\eps^1$ so that the expansion takes the form
\begin{equation}\label{asy_modif_ini}
\begin{array}{rl}
\displaystyle
u^\eps(x,t)
=
u^0(x,t)\!\!&\!\! \displaystyle +\eps\Big\{ \chi\Big(\frac{x}{\eps},\frac{t}{\eps^2}\Big)
+\chi_{\rm il}\Big(\frac{x}{\eps},\frac{t}{\eps^2}\Big)\Big\}\nabla u^0(x,t)
\\[2mm]
&\!\!  \displaystyle
+V^{\eps,1} +\eps^2 v^2\Big(x,t,\frac{x}{\eps},\frac{t}{\eps^2}\Big)
+\dots
\end{array}
\end{equation}
The initial layer type function $\chi_{\rm il}$ has been added in order to compensate the discrepancy in the initial condition. This function solves the following problem:
\begin{equation}\label{ini_aux}
\left\{ \begin{array}{lcl} \displaystyle
\frac{\partial\chi_{\rm il}}{\partial s}& = &\displaystyle \mathrm{div}\big( a(z,s)\nabla \chi_{\rm il}\big)   \\[4mm]
\chi_{\rm il}(0,z) & = & -\chi(0,z).
\end{array} \right.
\end{equation}
\begin{lemma}%{\cite[Theorem x.x]{P81}}
\label{l_inila}
The solution of problem  (\ref{ini_aux}) decays exponentially
as $s\to\infty$. We have
$$
\|\chi\big._{\rm il}(\cdot,s)\|\big._{L^\infty(\mathbb T^n)}\leq Ce^{-\nu s},\quad
\|\chi\big._{\rm il}\|\big._{L^\infty([s,s+1];H^1(\mathbb T^n))}\leq Ce^{-\nu s}
$$
\end{lemma}
\begin{proof}
The desired statement is an immediate consequence of  the fact that
$\int_{\mathbb T^n}\Big. \chi_{\rm il}(z,s)dz=\int_{\mathbb T^n} \chi_{\rm il}(z,0)dz=0$, the maximum principle
and the parabolic Harnack inequality (see \cite{KP02} for further details).
\end{proof}

\section{Main results}\label{sec6}
%-----------------

In this section we present the main result. Consider the expression
\begin{equation}\label{d_ucap0}
U^\eps(x,t)=\frac{u^\eps(x,t)-u^0(x,t)}{\eps}-\chi\Big(\frac x\eps,\frac{t}{\eps^2}\Big)\nabla_x u^0.
\end{equation}
It is easily seen that $U^\eps$ is equal to the normalized difference between $u^\eps$ and the first two terms of the asymptotic expansion.
The limit behaviour of $U^\eps$ is described by the following
statement.
\begin{thm}\label{t_main1}
Under the assumptions {\bf H1}--{\bf H3}
the function $U^\eps$ converges in law, as $\eps\to0$, in the space $L^2(\mathbb R^n\times(0,T))$ to a solution of the following SPDE
\begin{equation}\label{lim_spde}
\begin{array}{c}
\displaystyle
d U^0=\mathrm{div}\Big(a^{\rm eff}\nabla U^0+\mu \frac{\partial^3}{\partial x^3}u^0\Big)dt+\Lambda^{1/2}\frac{\partial^2}{\partial x^2}u^0\,dW_t,\\[4mm]
U^0(x,0)=0.
\end{array}
\end{equation}
\end{thm}
\begin{proof}
We set
$$
\begin{array}{c}
\displaystyle
\mathcal{V}^\eps(x,t)= U^\eps(x,t)-\eps^{-1}V^{\eps,1}(x,t)- \Xi_{\eps,2}(x,t)\\[3mm] \displaystyle
-\chi\big._{\rm il}\Big(\frac{x}{\eps},\frac{t}{\eps^2}\Big)\nabla u^0(x,t) -\eps\chi\big._{2,2}\Big(\frac{x}{\eps},\frac{t}{\eps^2}\Big) \frac{\partial^2}{\partial x^2}u^0(x,t)-\Xi_{\eps,1}(x,t).
\end{array}
$$
Substituting this expression in (\ref{ory_cauchy}) for $u^\eps$ and combining the above equations, we obtain after straightforward computations that $\mathcal{V}^\eps$ satisfies the problem
$$
\begin{array}{c}\displaystyle
\frac{\partial}{\partial t}\mathcal{V}^\eps-
\mathrm{div}\Big(a\Big(\frac{x}{\eps},\frac{t}{\eps^2}\Big)\nabla
\mathcal{V}^\eps\Big)=R^\eps,\\[3mm]
\displaystyle
\mathcal{V}^\eps(x,0)=R_1^\eps
\end{array}
$$
with
$$
\begin{array}{rl}
\displaystyle
R^\eps=\!&\!\!\eps^{-1}\big\{\partial_{z_i}[(a^\eps)^{ij}(\chi_{\rm il}^\eps)^k]+(a^\eps)^{ji}\partial_{z^i}(\chi_{\rm il}^\eps)^k)\big\}\partial_{x^j}\partial_{x^k}u^0-(\chi^\eps)^j\partial_t\partial_{x^j}u^0\\[2mm]
-\!&\!\!(\chi_{\rm il}^\eps)^j\partial_t\partial_{x^j}u^0+\eps(a^\eps)^{ij}(\chi_{2,2}^\eps)^{lk}\partial_{x^i}
\partial_{x^j}\partial_{x^l}\partial_{x^k}u^0-\eps(\chi_{2,2}^\eps)^{ij}\partial_t\partial_{x^i}\partial_{x^j}u^0
\end{array}
$$
and
$$
R_1^\eps=\eps\chi_{2,2}\Big(\frac{x}{\eps},0\Big)\partial_x\partial_x u^0(x,0).
$$
It follows from Lemma \ref{l_inila} that
$$
\big\|\eps^{-1}\partial_{z_i}[(a^\eps)^{ij}(\chi_{\rm il}^\eps)^k]\partial_{x^j}\partial_{x^k}u^0\big\|_{L^2(0,T);H^{-1}(\mathbb R^n))}
+\big\|(\chi_{\rm il}^\eps)^j\partial_t\partial_{x^j}u^0\big\|_{L^2((0,T)\times\mathbb R^n)}\leq C\eps,
$$
By (\ref{chi_est}), (\ref{norm_coo})  we obtain
$$
\big\|(\chi^\eps)^j\partial_t\partial_{x^j}u^0\big\|_{L^2((0,T);H^{-1}(\mathbb R^n))}\le C\eps.
$$
Then by (\ref{chi02_est}) we have
$$
\big\|\eps(a^\eps)^{ij}(\chi_{2,2}^\eps)^{lk}\partial_{x^i}
\partial_{x^j}\partial_{x^l}\partial_{x^k}u^0\big\|_{L^2((0,T)\times\mathbb R^n)}+
\big\|\eps(\chi_{2,2}^\eps)^{ij}\partial_t\partial_{x^i}\partial_{x^j}u^0\big\|_{L^2((0,T)\times\mathbb R^n)}\leq C\eps
$$
and
$$
\big\|\eps\chi_{2,2}\Big(\frac{x}{\eps},0\Big)\partial_x\partial_x u^0(x,0)\big\|_{L^2(\mathbb R^n)}\le C\eps.
$$
It remains to estimate the contribution of the term $\eps^{-1}(a^\eps)^{ji}\partial_{z^i}(\chi_{\rm il}^\eps)^k)\partial_{x^j}\partial_{x^k}u^0$.
%To this end we set
%$$
%\psi\big._{{\rm il},1}(s)=\int\limits_{\mathbb T^n}a(z,s)\partial_z  \chi_{\rm il}(z,s)\,dz, \qquad
%\psi\big._{{\rm il},2}(s)=a(z,s)\partial_z \psi_{\rm il}(z,s)-\chi\big._{{\rm il},1}(s)
%$$
From the estimates of Lemma \ref{l_inila} it is easy to deduce that
%$$
%\big\|\eps^{-1}\big((a^\eps)^{ji}\partial_{z^i}(\chi_{\rm il}^\eps)^k)-
%(\psi\big.^\eps_{{\rm il},1})^{jk}\big)\partial_{x^j}\partial_{x^k}u^0\big\|_{L^2((0,T);H^{-1}(\mathbb R^n))}\le C\eps.
%$$

$$
\big\|\eps^{-1}\big((a^\eps)^{ji}\partial_{z^i}(\chi_{\rm il}^\eps)^k\big)
\partial_{x^j}\partial_{x^k}u^0\big\|_{L^2((0,T)\times\mathbb R^n)}\leq C.
$$
and that a.s. the family $\big\{\eps^{-1}(a^\eps)^{ji}\partial_{z^i}(\chi_{\rm il}^\eps)^k)\partial_{x^j}\partial_{x^k}u^0\big\}$ converges to zero
weakly in $L^2((0,T)\times\mathbb R^n)$.
Then, using the same compactness arguments as those in the proof of Lemma \ref{l_small_h1} one can show that the solution of problem
$$
\left\{ \begin{array}{lcl} \displaystyle
\frac{\partial\Xi_{\eps,3}}{\partial t}& = &\displaystyle \mathrm{div}\Big( a\Big( \frac{x}{\eps}, \frac{t}{\eps^2} \Big)\nabla \Xi_{\eps,3}\Big)+
\eps^{-1}\big((a^\eps)\partial_{z}(\chi_{\rm il}^\eps)\big)\partial_{x}\partial_{x}u^0\ \\[4mm]
\Xi_{\eps,3} (0,x) & = & 0
\end{array} \right.
$$
converges a.s. to zero in $L^2((0,T)\times\mathbb R^n)$. Moreover,
 $$
\lim\limits_{\eps\to0}\E\big(\|\Xi_{\eps,3}\|^2_{L^2(\mathbb R^n\times[0,T])}\big)=0.
$$
Combining the above estimates we conclude that
$R^\eps$ a.s. tends to zero in $L^2(\mathbb R^n\times(0,T))$, as $\eps\to0$, and  $R_1^\eps$ a.s. tends to zero in $L^2(\mathbb R^n)$. Furthermore,
$$
\E\|R^\eps\|^2_{L^2(\mathbb R^n\times(0,T))}\,\to\, 0,\qquad
\E\|R_1^\eps\|^2_{L^2(\mathbb R^n)}\,\to\, 0.
$$

\medskip\noindent
By Lemmata \ref{l_small_h1} and \ref{l_inila} and estimate (\ref{chi02_est}) it follows that
$(U^\eps-\eps^{-1}V^{\eps,1}- \Xi_{\eps,2})$
tends a.s. to zero in $L^2(\mathbb R^n\times(0,T))$, and
$$
\E\|U^\eps-\eps^{-1}V^{\eps,1}(x,t)- \Xi_{\eps,2}(x,t)\|^2_{L^2(\mathbb R^n\times(0,T))}\,\to\,0.
$$
By Lemma \ref{l_stochpde} the function  $\eps^{-1}V^{\eps,1}$
converges in law to a solution of (\ref{aux_spde}). Also, $\Xi_{\eps,2}$ converges
a.s. to $\Xi_{0,2}$ in $L^2(\mathbb R^n\times(0,T))$.
This yields the convergence
$$
U^\eps \longrightarrow V^{0,1}+\Xi_{0,2}
$$
in law in the space $L^2(\mathbb R^n\times(0,T))$.
It remains to note that due to (\ref{aux_spde}) and (\ref{eq_homaux}) the random function $U^0:=(V^{0,1}+\Xi_{0,2})$ satisfies the stochastic PDE (\ref{lim_spde}) as required.
\end{proof}

\section{Diffusive case}\label{sec7_dif}

The goal of this section is to extend the statement of Theorem \ref{t_main1} to the diffusive case.

\begin{thm}\label{t_main1_dif}
Let assumptions  {\bf A1}--{\bf A3} be fulfilled.
Then the function $U^\eps$ defined in (\ref{d_ucap0}) converges in law, as $\eps\to0$, in the space $L^2(\mathbb R^n\times(0,T))$ to the solution of
(\ref{lim_spde})
%the following SPDE
%\begin{equation}\label{lim_spde_dif}
%\begin{array}{c}
%\displaystyle
%d U^0=\mathrm{div}\Big(a^{\rm eff}\nabla U^0+\mu \frac{\partial^3}{\partial x^3}u^0\Big)dt+\Lambda^{1/2}\frac{\partial^2}{\partial
%x^2}u^0\,dW_t,\\[4mm]
%U^0(x,0)=0.
%\end{array}
%\end{equation}
\end{thm}
\begin{proof}
The arguments used in the proof of Theorem \ref{t_main1} also apply in the case under consideration.  We  used assumption {\bf H3} only once, when
 justified convergence (\ref{inv_p}).  Thus, this convergence should be reproved under our standing assumptions.

\begin{lmm}\label{l_mixi_dif}
Under assumptions {\bf A1}--{\bf A3} for any $K>0$ there exists $C_K$ such that the following estimate holds
$$
\|\E\{\Psi_{2,1}(s)\,|\, \mathcal{F}_{\leq 0}\}\|\big._{L^2(\Omega)}\leq C_K\big(e^{-\nu s/2}+(1+s)^{-K}\big), \qquad
\nu>0Б
$$
the function $\Psi_{2,1}$ has been defined in (\ref{defpsi21})
\end{lmm}
\begin{proof}
We follow the scheme of proof of Lemma \ref{l_mixi}.
Denote
$$
\chi(z,s)=\widehat\chi^1(z,s)+\widehat\chi^2(z,s)=
$$
$$
\int\limits_{-\infty}^{s/2}\int\limits_{\mathbb T^n}\big(\mathcal{G}(z,\hat z,s,\hat s)-1\big)\mathrm{div}_za(\hat z,\xi_{\hat s})\,d\hat zd\hat s
+\int\limits^{s}_{s/2}\int\limits_{\mathbb T^n}\big(\mathcal{G}(z,\hat z,s,\hat s)\big)\mathrm{div}_za(\hat z,\xi_{\hat s})\,d\hat zd\hat s.
$$
Then $\Psi_ {2,1}(s)=\widehat\Psi^1(s)+\widehat\Psi^2(s)$ with
$$
 \widehat\Psi^i(s)=\int_{\mathbb T^n}\big(a(z,\xi_s)\nabla_z\widehat\chi^i(z,s)-\E(a(z,\xi_s)\nabla_z\widehat\chi^i(z,s))\big)\,dz, \quad i=1,\,2.
$$
Considering (\ref{green_est1}) we obtain the inequality $\|\widehat\Psi^1(s)\|_{L^2(\Omega)}\leq Ce^{-\nu s/2}$.  Since $\widehat\Psi^2(s)$ is $ \mathcal{F}_{\geq s/2}$~-~measurable,
we have
$$
\|\E\{\widehat\Psi^2(s)\,|\, \mathcal{F}_{\leq 0}\}\|\big._{L^2(\Omega)}= \|\E\big\{\E\{\widehat\Psi^2(s)\,|\, \mathcal{F}_{\leq s/2}\}\,|\, \mathcal{F}_{\leq 0}\big\}\|\big._{L^2(\Omega)}
$$
$$
= \|\E\big\{\E\{\widehat\Psi^2(s)\,|\, \mathcal{F}_{= s/2}\}\,|\, \mathcal{F}_{\leq 0}\big\}\|\big._{L^2(\Omega)}=
 \|\E\big\{\mathcal{R}(\xi_{s/2})\,|\, \mathcal{F}_{\leq 0}\big\}\|\big._{L^2(\Omega)};
$$
here we have used the Markov property of $\xi_\cdot$.
According to \cite[Section 2]{PV} this yields the desired inequality.
\end{proof}

\bigskip\noindent
From the last Lemma it follows that the invariance principle holds for the process $\chi_{2,1}(s)$  (see \cite[Theorem VIII.3.79]{JSh}),
that is (\ref{inv_p}) holds for any $T>0$.
The rest of proof of Theorem \ref{t_main1_dif} is exactly the same as that of Theorem \ref{t_main1}.
\end{proof}

\end{document}